\documentclass{amsart}

\usepackage{amssymb}
\usepackage{amsrefs}
\usepackage{enumerate}

\numberwithin{equation}{section}
\newtheorem*{thm}{Theorem}
\newtheorem{lemma}{Lemma}
\newtheorem{prop}[lemma]{Proposition}

{%
\theoremstyle{remark}
\newtheorem*{rem}{Remark}
}

\newcommand{\Gal}{\mathop{\rm Gal}\nolimits}

\newcommand{\lhs}{\mathop{\rm LHS}\nolimits}
\newcommand{\GL}{\mathop{\rm GL}\nolimits}
\newcommand{\DL}{\mathop{\rm R}\nolimits}

\newcommand{\EE}{{\mathcal{E}}}
\newcommand{\NN}{{\mathcal{N}}}
\newcommand{\map}[3]{#1\colon #2 \longrightarrow #3}
\newcommand{\conn}{^\circ}

\newcommand{\inv}{^{-1}}

\newcommand{\autgp}{\langle\varepsilon\rangle} 
\let\tilde\widetilde
\newcommand{\lsup}[1]{{}^{#1}}

\newcommand{\set}[2]{ 
	{\bigl\{\bigl.
	#1\vphantom{#2  
		}\,\bigr|
	\,#2\bigr\}}}

\author[Adler]{Jeffrey D.~Adler}
\address{Department of Mathematics and Statistics\\
American University\\
Washington, DC 20016-8050}
\email[Adler]{jadler@american.edu}

\author[Cassel]{Michael Cassel}
\curraddr[Cassel]{Columbia University \\
116th St. and Broadway, \\
New York, NY 10027}
\email{m.cassel@columbia.edu}

\author[Lansky]{Joshua M.~Lansky} 
\email[Lansky]{lansky@american.edu}

\author[Morgan]{Emma Morgan}
\curraddr[Morgan]{Office of Institutional Research and Evaluation \\
Tufts University \\
Medford, MA 02155}
\email{emma.morgan@tufts.edu}

\author[Zhao]{Yifei Zhao}
\curraddr[Zhao]{Department of Mathematics \\
Harvard University \\
Cambridge, MA 02138}
\email{yifei@math.harvard.edu}

\title[Character relation]{Lifting representations of finite reductive groups: a character relation}
\date{\today}

\begin{document}

\begin{abstract}
Given a connected reductive group $\tilde{G}$ over a finite field $k$,
and a semisimple $k$-automorphism $\varepsilon$ of $\tilde{G}$ of finite order,
let $G$ denote the connected part of the group of $\varepsilon$-fixed points.
Then two of the authors have previously shown that there exists a natural lifting from series of
representations of $G(k)$ to series for $\tilde{G}(k)$.
In the case of Deligne-Lusztig representations, we show that this lifting satisfies
a character relation analogous to that of Shintani.
\end{abstract}

\maketitle

\addtocounter{section}{-1}
\section{Introduction}
\label{sec:intro}
Suppose $k$ is a finite field,
$\tilde{G}$ is a connected reductive $k$-group,
and $\varepsilon$ is a semisimple $k$-automorphism of $\tilde{G}$
of finite order $\ell$.
Let $G$ be the connected part of the group $\tilde{G}^\varepsilon$
of $\varepsilon$-fixed points of $\tilde{G}$.
We will see (Proposition \ref{prop:intro})
that $G$ is also a connected reductive $k$-group.
Two of the authors have constructed a natural lifting
\cite{adler-lansky:lifting}
from series of irreducible representations of $G(k)$ to analogous
series for $\tilde{G}(k)$.
In certain cases,
this lifting is known to coincide with Shintani lifting
(see \cite{gyoja:liftings}*{Thm.~7.2}, \cite{digne:arcata}*{Cor.~3.6}, or \cite{silberger-zink:level-zero-matching}*{Prop.~B4.4}, for example),
so it is natural to ask whether there is a relation, analogous to that of Shintani,
between the character of a representation
$\pi$ of $G(k)$ and the $\varepsilon$-twisted character of its lift $\tilde\pi$.
The purpose of this note is to prove an affirmative answer in the case where $\pi$
is a Deligne-Lusztig representation, irreducible or not.

Let $\tilde{G}^*$ and $G^*$ denote the duals of $\tilde{G}$ and $G$.
For each semisimple element $z\in G^*(k)$, one obtains a collection
$\EE_z(G(k))$ of irreducible representations of $G(k)$,
and these collections, known as \emph{rational Lusztig series},
partition
the set $\EE(G(k))$ of (equivalence classes of)
irreducible representations of $G(k)$
\cite{lusztig:chars-finite}*{\S14.1}.
For example,
suppose that $z$ is regular in $G^*$,
and let $T^*\subseteq G^*$ be the unique maximal $k$-torus
containing $z$.
Then the pair $(T^*,z)$ corresponds to a pair $(T,\theta)$, where
$T\subseteq G$ is a maximal $k$-torus, and $\theta$ is a complex-valued character of $T(k)$.
This latter pair is uniquely determined up to $G(k)$-conjugacy.
The Lusztig series $\EE_z(G(k))$ corresponding to $z$ is the set of irreducible components
of the Deligne-Lusztig virtual representation whose character is $\DL_T^G \theta$.
An earlier work
\cite{adler-lansky:lifting}*{Corollary 11.3}
presents a natural
map from semisimple conjugacy classes in $G^*(k)$
to semisimple classes in $\tilde{G}^*(k)$,
thus lifting
each Lusztig series for $G(k)$ to one for $\tilde{G}(k)$.
The series of representations coming from $\pm \DL_T^G\theta$ lifts to that
coming from $\pm \DL_{\tilde{T}}^{\tilde{G}} \tilde\theta$,
where we will see that $\tilde{T} := C_{\tilde{G}}(T)$
is a maximal $k$-torus in $\tilde{G}$,
and $\tilde\theta = \theta \circ \NN$,
where $\map\NN{\tilde{T}}T$ is the norm map defined by
\[
\NN(t) = t \varepsilon(t) \cdots \varepsilon^{\ell-1}(t).
\]
(Of course, one could define a similar map $\NN$ on any $\varepsilon$-invariant torus in $\tilde{G}$.)
In order to understand better this lifting of representations,
one would like to have a relation between the character $\DL_T^G\theta$
(for $\theta$ an arbitrary character of $T(k)$, not necessarily associated to a regular
element of $G^*(k)$)
and the $\varepsilon$-twisted character $(\DL_{\tilde{T}}^{\tilde{G}}\tilde\theta)_\varepsilon$
associated to the
$\varepsilon$-invariant character $\DL_{\tilde{T}}^{\tilde{G}}\tilde\theta$.
We prove that such a relation 
holds at sufficiently regular points,
provided that one gathers together a ``packet'' of characters of $G(k)$,
including
$\DL_T^G\theta$,
that each lift to $\DL_{\tilde{T}}^{\tilde{G}}\tilde\theta$.

Suppose we have a 
\emph{Borel-torus pair} $(B,T)$ for $G$.
That is, we have
a maximal $k$-torus $T\subseteq G$,
together with a Borel subgroup $B\subseteq G$ containing $T$,
not necessarily defined over $k$.
From Proposition \ref{prop:intro}, we will obtain
an $\varepsilon$-invariant Borel-torus pair $(\tilde{B},\tilde{T})$
in $\tilde{G}$,
where $\tilde{T} = C_{\tilde{G}}(T)$ as above.
Let $\autgp$ be the group generated by $\varepsilon$.
Recall that an element of a reductive group is \emph{regular} if the connected
part of its centralizer is a torus.
\begin{thm}
Suppose $\tilde{s} \in \tilde{G}(k)$ belongs to an $\varepsilon$-invariant, maximal $k$-torus $\tilde S$,
and that 
$\NN(\tilde{s})$ is regular
in $\tilde{G}$.
Then
\begin{equation}
\label{eqn:main1}
(\DL_{\tilde{T}}^{\tilde{G}}\tilde{\theta})_{\varepsilon}(\tilde{s})
=
\left(
	\sum_{\tilde g \in G(k) \backslash (\tilde{G}(k) / \tilde{T}(k))^\varepsilon}
	\DL_{\lsup{\tilde g}T}^G \lsup{\tilde g}\theta
\right)
(\NN(\tilde{s})).
\end{equation}
\end{thm}

\begin{rem}
Let us comment on some of the terms in Equation \eqref{eqn:main1}.
\begin{enumerate}[(a)]
\item
Here is what we mean by
$(\DL_{\tilde{T}}^{\tilde{G}}\tilde{\theta})_{\varepsilon}$.
Since 
$\tilde{\theta}$ is an $\varepsilon$-invariant character of $\tilde{T}(k)$,
we have that $\varepsilon$ 
acts on the Deligne-Lusztig variety corresponding to $(\tilde{B},\tilde{T},\tilde{\theta})$,
and thus on the virtual representation
whose character is
$\DL_{\tilde T}^{\tilde G} \tilde\theta$.
That is, even if this representation is reducible, we can form its $\varepsilon$-twisted character.
To do so, extend $\tilde\theta$ to a character of $\tilde{T}(k)\rtimes\autgp$
by setting
$\tilde\theta(\varepsilon) = 1$.
Define the $\varepsilon$-twisted Deligne-Lusztig character
$(\DL_{\tilde{T}}^{\tilde{G}}\tilde{\theta})_{\varepsilon}$ induced from $\tilde{\theta}$ by
$(\DL_{\tilde{T}}^{\tilde{G}}\tilde{\theta})_{\varepsilon}(g) = 
(\DL_{\tilde{T}\rtimes\autgp }
	^{\tilde{G}\rtimes\autgp}\tilde\theta)(g\varepsilon)$
for $g\in\tilde{G}(k)$.
(See~\cite{digne-michel:non-connected} for the definition of Deligne-Lusztig induction for nonconnected groups.)

Note that if $\varepsilon$ is quasi-central,
i.e., the Weyl groups $W(G,T)$ and $W(\tilde{G},\tilde T)$ satisfy
$W(G,T) = W(\tilde{G},\tilde T)^\varepsilon$
(see~\cite{digne-michel:non-connected}*{D\'efinition-Th\'eor\`eme 1.15} for equivalent formulations),
then
$(\DL_{\tilde{T}}^{\tilde{G}}\tilde{\theta})_{\varepsilon}$
doesn't depend on the choice of $\tilde{B}$, from
the remark after
\cite{digne-michel:non-connected}*{Th\'eor\`eme 4.5}.
More generally, there could be several Borel subgroups $\tilde{B} \subseteq \tilde{G}$
such that $(\tilde{B}^\varepsilon)\conn= B$,
and we don't know if the twisted character is independent of the choice of $\tilde{B}$.
However, our theorem will remain valid for any such choice.

\item
On the right-hand side of Equation \eqref{eqn:main1},
$\tilde g$ runs over a set of double coset representatives.  For each such $\tilde g$,
we have that $\lsup{\tilde g}T$ is a maximal $k$-torus in $G$ and $\lsup{\tilde g} \theta$
is a complex-valued character of $\lsup{\tilde g}T(k)$.
The choice of representative $\tilde g$ affects the pair $(\lsup{\tilde g}T, \lsup{\tilde g}\theta)$
only up to $G(k)$-conjugacy, so it does not affect the character
$\DL_{\lsup{\tilde g}T}^G \lsup{\tilde g}\theta$
appearing in the corresponding summand.

\item
We note that $G(k) \backslash (\tilde{G}(k) / \tilde{T}(k))^\varepsilon$ can be viewed as the
set of $G(k)$-conjugacy classes of $\varepsilon$-invariant Borel-torus pairs
that are $\tilde G(k)$-conjugate to $(\tilde B,\tilde T)$.
As $\tilde{g}$ runs over a set of double coset representatives,
the $k$-tori $\lsup{\tilde{g}} T \subset G$ are related in a way that is analogous
to stable conjugacy,
where $\tilde{G}(k)$ plays the role usually played by $G(\overline k)$,
where $\overline k$ is an algebraic closure of $k$.
Moreover, the set of characters appearing in the right-hand side of \eqref{eqn:main1}
is analogous to an $L$-packet,
and their sum is ``stable'', in the sense that it is constant
on the intersections with $G(k)$ of appropriate conjugacy classes in $\tilde{G}(k)$.
\end{enumerate}
\end{rem}

Now let's consider some special cases.
\begin{enumerate}[(a)]
\item
\label{item:restriction-of-scalars}
One can show that when $\varepsilon$ is quasi-central, the index set for the summation is a singleton,
so the theorem asserts that
$(\DL_{\tilde{T}}^{\tilde{G}}\tilde{\theta})_{\varepsilon}(\tilde{s})
=(\DL_{T}^G \theta)(\NN(\tilde{s}))$.
\item
In particular,
suppose $\tilde{G} = R_{E/k} G$ is obtained from $G$ via restriction of scalars over a cyclic extension $E/k$ and
$\varepsilon$ is the algebraic $k$-automorphism of $\tilde{G}$ associated to the action
of a generator of the Galois group $\Gal(E/k)$.
Given a representation $\pi$ of $G(k)$, one often has an associated
representation $\tilde\pi$ of $\tilde{G}(k)$, known as the \emph{Shintani lift}
of $\pi$.
(See \cite{kawanaka:shintani} for a discussion.)
The character of $\pi$ and the $\varepsilon$-twisted character of $\tilde\pi$
satisfy the Shintani relation:
$\Theta_{\tilde\pi,\varepsilon}
=
\Theta_\pi\circ\NN$,
where $\NN$ has been extended to a map on all
(not necessarily semisimple) conjugacy classes.
From work of Digne \cite{digne:shintani}*{Cor.\ 3.6},
one already knows that if $\DL_T^G \theta$ has a Shintani lift, then it must be
$\DL_{\tilde T}^{\tilde G} \tilde\theta$.
Thus, if one restricts one's attention to Deligne-Lusztig characters
and to the kind of elements that we consider,
our character relation is a generalization of Shintani's.
\item
Consider the automorphism of $\GL(2)$ given by
$\varepsilon(g) = {}^t g \inv$.
Then a relation analogous to that in the Theorem holds for all irreducible
characters, not just those of Deligne-Lusztig type.
Moreover, the relation can be extended to unipotent elements,
but it fails if
$\tilde{s}$ is regular but $\NN(\tilde s)$ is singular.
\end{enumerate}

Comparing our theorem with \cite{digne-michel:non-connected}*{Proposition 2.12}, we see that the latter shows
that the restriction to $G(k)$ of a twisted character of $\tilde{G}(k)$
is a certain twisted character of $G(k)$, at least in the case where $\varepsilon$ is quasi-central.
On the other hand, our theorem concerns a lifting of characters from $G(k)$ to $\tilde{G}(k)$ via the norm map,
and this lifting
is not an inverse of restriction.
As expected, the two formulas agree in those situations where both are applicable, but such situations are rare.
Moreover, while we make regularity assumptions on our element $\tilde{s}$, we do not assume that $\varepsilon$
is quasi-central.

\textsc{Acknowledgements.}
The authors were partially supported by a 
National Science Foundation Focused Research Group grant (DMS-0854844). 
The fourth-named author was partially supported by 
the Frank Cox Jones Scholarship Fund and
a Dean's Undergraduate Research Award from the College of Arts and Sciences of American
University.
We have benefited from conversations with Prof.\ Jeffrey Hakim,
and the comments of an anonymous referee.

\section{Preliminary results}
Given any algebraic group $H$, we denote by $H\conn$ the connected component of the identity in $H$. If $S$ is an algebraic subgroup of $H$, then $C_H(S)$ denotes the centralizer of $S$ in $H$.
Similarly, for an element $h\in H$, we let $C_H(h)$ denote its centralizer.

As in~\S\ref{sec:intro},
let $k$ denote a finite field,
$\tilde G$ a connected reductive $k$-group,
$\varepsilon$ a semisimple $k$-automorphism of $\tilde G$
of finite order $\ell$,
and
$G = (\tilde{G}^\varepsilon)\conn$.

Most of the following result appears in the statement or proof of \cite{steinberg:endomorphisms}*{Theorem 8.2}.
Other versions appear in 
\cite{kottwitz-shelstad:twisted-endoscopy}*{Theorem 1.1.A}
and
\cite{digne-michel:non-connected}*{Th\'eor\`eme 1.8}.
A version that includes the rationality of $G$ is in
\cite{adler-lansky:lifting}*{Proposition 3.5}.
\begin{prop}
\label{prop:intro}
With notation as above, one has:
\begin{itemize}
\item
$G$ is a connected reductive $k$-group.
\item
For every $\varepsilon$-invariant Borel-torus pair $(\tilde{B},\tilde{T})$
for $\tilde{G}$,
one has a Borel-torus pair $((\tilde{B}^\varepsilon)\conn,(\tilde{T}^\varepsilon)\conn)$
for $G$.
Moreover,
$(\tilde{T}^\varepsilon)\conn = \tilde{T} \cap G$.
\item
For every Borel-torus pair $(B,T)$ for $G$, 
one has an $\varepsilon$-invariant Borel-torus pair $(\tilde{B},\tilde{T})$
where $\tilde{T} = C_{\tilde{G}}(T)$, and such that $(\tilde{B}^\varepsilon)\conn = B$.
\end{itemize}
\end{prop}

From now on, we
fix a maximal $k$-torus $T\subseteq G$, and thus obtain
an $\varepsilon$-invariant maximal $k$-torus $\tilde{T} = C_{\tilde{G}}(T)$
as in Proposition \ref{prop:intro},
and a norm map $\map\NN{\tilde{T}}T$ as in \S\ref{sec:intro}.

The following result
concerns conjugacy and twisted conjugacy in $G$ and $\tilde G$.
\begin{lemma}
\label{lem:twisted-conjugacy}
Suppose $\tilde{S}$ is an 
$\varepsilon$-invariant maximal $k$-torus in $\tilde{G}$.
Let $S = \tilde S\cap G$.
Suppose that $\tilde s \in \tilde{S}(k)$,
and that $s:= \NN(\tilde{s})\in S(k)$ is regular in $\tilde{G}$.
Let $\tilde g \in \tilde{G}(k)$.
\begin{enumerate}[(i)]
\item
\label{item:in-borus}
There is an $\varepsilon$-invariant Borel subgroup of $\tilde{G}$ containing $\tilde{S}$.
\item
\label{item:transport1}
If $\tilde g\inv s\tilde g \in T(k)$, then $\tilde g\inv \tilde S\tilde g = \tilde T$.
\item
\label{item:transport2}
If $\tilde g\inv \tilde s\tilde g \in \tilde T(k)$ and $\tilde g\inv\varepsilon(\tilde g)\in \tilde T(k)$, then
$\tilde g\inv  S \tilde g = T$.
\item
\label{item:transport3}
If $\tilde g\inv \tilde s \varepsilon(\tilde g) \in \tilde T(k)$, then
$\tilde g\inv\varepsilon(\tilde g)\in \tilde T(k)$.
\end{enumerate}
\end{lemma}

\begin{proof}
To prove (\ref{item:in-borus}),
let $S'$ be a maximal $k$-torus in $G$ such that $s\in S'(k)$,
and let $\tilde{S}' = C_{\tilde{G}}(S')$.
From Proposition \ref{prop:intro}, $\tilde{S}'$ is a maximal torus
in $\tilde{G}$, and it is contained in an $\varepsilon$-invariant Borel subgroup of $\tilde{G}$.
But
$
\tilde{S}' = C_{\tilde{G}}(S') \subseteq C_{\tilde{G}}(s)\conn = \tilde{S},
$
and therefore, $\tilde{S}' = \tilde{S}$.

To prove (\ref{item:transport1}),
note that
$
\tilde g\inv \tilde S \tilde g = \tilde g\inv C_{\tilde G}(s)\conn \tilde g
= C_{\tilde G}(\tilde g\inv s \tilde g)\conn
= \tilde T
$.

For (\ref{item:transport2}), note that it follows immediately from the assumptions that 
$\tilde g\inv\tilde s \varepsilon(\tilde g)\in\tilde T(k)$.
Then
$\tilde g\inv s\tilde g = \NN(\tilde g\inv\tilde s \varepsilon(\tilde g))\in \NN(\tilde T(k))\subseteq T(k)$.
From (\ref{item:transport1}),
$\tilde g\inv \tilde S\tilde g = \tilde T$.
From (\ref{item:in-borus}) and Proposition \ref{prop:intro},
it is thus enough to show that
$(\tilde g\inv \tilde S\tilde g)^\varepsilon = \tilde g\inv \tilde S^\varepsilon \tilde g$.
For $x \in \tilde{S}(\overline k)$,
\begin{align*}
\varepsilon(\tilde g \inv x  \tilde g)
&= 
\varepsilon(\tilde g) \inv \varepsilon(x)  \varepsilon(\tilde g) \\
&=
[\tilde g \inv  \varepsilon(\tilde g) ] \inv
\cdot
\tilde g \inv \varepsilon(x) \tilde g
\cdot
[\tilde g\inv  \varepsilon(\tilde g)] \\
&= 
\tilde g \inv \varepsilon(x) \tilde g,
\end{align*}
so $\varepsilon(x) = x$ if and only if $\varepsilon(\tilde g\inv x \tilde g) = \tilde g\inv x \tilde g$.

To prove (\ref{item:transport3}), note that as above, 
$\tilde g\inv\tilde s \varepsilon(\tilde g)\in\tilde T(k)$
implies $\tilde g\inv s\tilde g \in T(k)$.
By (\ref{item:transport1}), we have $\tilde g\inv \tilde s\tilde g\in\tilde T(k)$.
Therefore,
\[
\tilde g\inv\varepsilon(\tilde g)
= (\tilde g\inv \tilde s\tilde g)\ (\tilde g\inv\tilde s \varepsilon(\tilde g))\inv \in \tilde T(k).
\qedhere
\]
\end{proof}

Now we consider a property of certain double coset spaces.
\begin{lemma}
\label{lemma:double-cosets}
Let $\tilde R$ denote a set of representatives for the double coset space
$$
G(k) \backslash (\tilde{G}(k)/\tilde{T}(k))^\varepsilon.
$$
Define a map
$\map{\phi}{G(k) \times \tilde R}{(\tilde{G}(k)/\tilde{T}(k))^\varepsilon}$
by
$$
(g,\tilde{r}) \longmapsto g \tilde{r} \tilde{T}(k).
$$
Then
\begin{enumerate}[(i)]
\item
The map $\phi$ is surjective.
\item
We have that
$\phi(g, \tilde{r}) = \phi(g', \tilde{r}')$
if and only if
$\tilde{r}=\tilde{r}'$
and $g\inv g' \in \lsup {\tilde{r}} T(k)$.
\end{enumerate}
\end{lemma}

\begin{proof}
This is straightforward.
\end{proof}

Now we recall some facts about Deligne-Lusztig virtual characters.
\begin{lemma}
\label{lemma:char-formulas}
One has the following:
\begin{enumerate}[(i)]
\item
\label{item:char-downstairs}
If $x\in G(k)$ is a regular element
and $\theta$ is a complex character of $T(k)$,
then
\[
(\DL_{T}^{G}\theta)(x)=
\sum_{\substack{g\in G(k)/T(k)\\ g\inv x g\in T(k)}}\theta(g\inv xg).
\]
\item
\label{item:char-upstairs}
Suppose that $\tilde{x}\in\tilde G(k)$ and that $\tilde x\varepsilon$
is a regular element of $\tilde G(k)\rtimes\autgp$.  Let $\tilde{\theta}$
be an $\varepsilon$-invariant character of $\tilde{T}(k)$, extended trivially to
$\tilde{T}(k) \rtimes\autgp$.
Then
\[
(\DL_{\tilde{T}}^{\tilde{G}}\tilde{\theta})_{\varepsilon}(\tilde{x})
=
\frac{1}{\vert\tilde{T}(k)\rtimes\autgp \vert}
	\sum_{\substack{
		\tilde h\in\tilde{G}(k)\rtimes\autgp \\
		\tilde{x}\varepsilon\in \tilde h(\tilde{T}(k)\rtimes\autgp)\tilde h\inv
	}}
	\tilde\theta (\tilde h\inv(\tilde{x}\varepsilon)\tilde h) .
\]
\end{enumerate}
\end{lemma}

\begin{proof}
For each formula, see \cite{digne-michel:non-connected}*{Proposition 2.6}.
For the second formula,
note that $C_{\tilde G}(\tilde x\varepsilon)$
contains no nontrivial unipotent elements.
Thus, if we let $S$ denote its connected part,
then $S$ is a torus, and
and the Green function
$Q_{S(k)}^{S(k)}$
that arises in this proposition takes the value $|S(k)|$
at $(1,1)$.
\end{proof}

\section{Proof of the Main Theorem}

From Lemma \ref{lem:twisted-conjugacy}(\ref{item:in-borus}),
$\tilde{S}$ is contained in an $\varepsilon$-invariant Borel subgroup of $\tilde{G}$,
so it follows from Proposition \ref{prop:intro} that
$S := \tilde{S} \cap G$ is 
a maximal $k$-torus in $G$, 
and $\tilde{S} = C_{\tilde{G}}(S)$.

Since $(\tilde{s}\varepsilon)^\ell = \NN(\tilde s)$, which is assumed regular in $\tilde{G}$ and thus
in $\tilde{G}\rtimes \autgp$,
we must have that $\tilde{s}\varepsilon$ is also regular in $\tilde G\rtimes\autgp$.
Lemma~\ref{lemma:char-formulas}(\ref{item:char-upstairs}) then implies that 
the left-hand side ($\lhs$) of 
Equation \eqref{eqn:main1} 
is equal to
\[
\frac{1}{\ell|\tilde{T}(k)|}\sum \tilde\theta (\tilde h\inv(\tilde{s}\varepsilon) \tilde h)
=
\frac{1}{\ell|\tilde{T}(k)|}\sum \tilde\theta
	\left(\tilde h\inv\cdot\tilde{s}\cdot\varepsilon (\tilde h)\cdot\varepsilon\right),
\]
where each sum is over the set
$\set{\tilde h\in \tilde{G}(k)\rtimes\autgp }{
\tilde{s}\varepsilon \in \tilde h \bigl(\tilde{T}(k)\rtimes \autgp \bigr) \tilde h\inv}$.
If $\tilde g\in \tilde{G}(k)$, then $\tilde g\varepsilon^i$ belongs to the index set
if and only if
$\tilde g\inv\tilde{s}\varepsilon(\tilde g)\in \tilde{T}(k)$.
Note that the set of such elements $\tilde g$ is a union of left cosets of $\tilde T(k)$.
Thus,
\begin{align*}
\lhs
&=
\frac{1}{\ell|\tilde{T}(k)|}
\sum_{i=0}^{\ell-1}
\sum_{\substack{\tilde g \in \tilde G(k)\\ \tilde g\inv\tilde{s}\varepsilon(\tilde g)\in \tilde{T}(k)}}
\tilde\theta \left(
	(\tilde g\varepsilon^i)\inv\cdot\tilde{s}\cdot\varepsilon (\tilde g\varepsilon^i) \cdot\varepsilon
	\right) \\
&=
\frac{1}{|\tilde{T}(k)|}
\sum_{\substack{\tilde g \in \tilde G(k)\\ \tilde g\inv\tilde{s}\varepsilon(\tilde g)\in \tilde{T}(k)}}
\tilde\theta (\tilde g\inv \tilde s \varepsilon(\tilde g)) \\
&=
\sum_{\substack{
	\tilde g \in \tilde G(k)/\tilde T(k)\\ \tilde g\inv\tilde{s}\varepsilon(\tilde g)\in \tilde{T}(k)
	}}
\tilde\theta (\tilde g\inv \tilde s \varepsilon(\tilde g)) \\
&=
\sum_{\substack{
	\tilde g \in (\tilde{G}(k)/\tilde T(k))^\varepsilon\\ \tilde g\inv\tilde{s}\tilde g\in \tilde{T}(k)
	}}
\tilde\theta (\tilde g\inv \tilde s \varepsilon(\tilde g)) ,
\end{align*}
where the last equality follows from Lemma~\ref{lem:twisted-conjugacy}(\ref{item:transport3}).
Let $s = \NN(\tilde s)$ and note that $\NN(\tilde g\inv \tilde s \varepsilon(\tilde g)) = \tilde g\inv s\tilde g$
for $\tilde g\in\tilde G(k)$.
Thus
$\lhs$
is equal to
\begin{equation}
\label{eq:lhs}
\sum_{\substack{
	\tilde g \in (\tilde{G}(k)/\tilde T(k))^\varepsilon\\ \tilde g\inv\tilde{s}\tilde g\in \tilde{T}(k)
	}}
\lsup{\tilde g}\theta (s) .
\end{equation}

On the other hand, it follows from Lemma~\ref{lemma:char-formulas}(\ref{item:char-downstairs})
that the right-hand side in Equation \eqref{eqn:main1} is equal to
\begin{align*}
\sum_{\tilde x \in G(k) \backslash (\tilde{G}(k) / \tilde{T}(k))^\varepsilon}
	\left(\DL_{\lsup{\tilde x}T}^G \lsup{\tilde x}\theta \right)
(s)
&=
	\sum_{\tilde x \in G(k) \backslash (\tilde{G}(k) / \tilde{T}(k))^\varepsilon}
	\ \sum_{\substack{g\in G(k)/\lsup{\tilde x}T(k)\\ g\inv s g\in\lsup{\tilde x} T(k)}}
	\lsup{\tilde x}\theta (g\inv s g)\\
&=
	\sum_{\tilde x \in G(k) \backslash (\tilde{G}(k) / \tilde{T}(k))^\varepsilon}
	\sum_{\substack{g\in G(k)/\lsup{\tilde x}T(k)\\ (g\tilde x)\inv s (g\tilde x)\in T(k)}}
	\lsup{g\tilde x}\theta (s)\\
&=
	\sum_{\substack{
		\tilde g \in (\tilde{G}(k)/\tilde T(k))^\varepsilon\\ \tilde g\inv s \tilde g\in T(k)
	}}
	\lsup{\tilde g}\theta (s),
\end{align*}
where the final equality follows from Lemma \ref{lemma:double-cosets}.
The last sum above is equal to \eqref{eq:lhs} by
Lemma~\ref{lem:twisted-conjugacy}(\ref{item:transport1}) and (\ref{item:transport2}),
which together imply that for $\tilde g\in (\tilde G(k)/\tilde T(k))^\varepsilon$,
$\tilde g\inv s \tilde g\in T(k)$ if and only if $\tilde g\inv\tilde{s}\tilde g\in \tilde{T}(k)$.

\end{document}